\documentclass{amsart}

\usepackage{amssymb}
\usepackage[hidelinks]{hyperref}
\usepackage{todonotes}

\theoremstyle{plain}
\newtheorem{theorem}{Theorem}[section]
\newtheorem{lemma}[theorem]{Lemma}
\newtheorem{corollary}[theorem]{Corollary}
\newtheorem{proposition}[theorem]{Proposition}

\theoremstyle{definition}
\newtheorem{definition}[theorem]{Definition}

\newtheorem{example}[theorem]{Example}

\newtheorem{question}{Question}

\theoremstyle{remark}

\newtheorem*{claim}{Claim}

\newtheorem*{acknowledgment}{Acknowledgments}

\newcommand{\bR}{\mathbb{R}}
\newcommand{\R}{\bR}

\newcommand{\bSigma}{\mathbf{\Sigma}}
\newcommand{\bPi}{\mathbf{\Pi}}

\newcommand{\cA}{\mathcal{A}}
\newcommand{\cB}{\mathcal{B}}
\newcommand{\cC}{\mathcal{C}}

\newcommand{\cD}{\mathcal{D}}

\newcommand{\cF}{\mathcal{F}}
\newcommand{\cG}{\mathcal{G}}
\newcommand{\cH}{\mathcal{H}}
\newcommand{\cI}{\mathcal{I}}
\newcommand{\I}{\cI}
\newcommand{\cJ}{\mathcal{J}}
\newcommand{\J}{\cJ}
\newcommand{\cM}{\mathcal{M}}

\newcommand{\cP}{\mathcal{P}}

\newcommand{\cS}{\mathcal{S}}
\newcommand{\cU}{\mathcal{U}}

\newcommand{\continuum}{\mathfrak{c}}
\newcommand{\bnumber}{\mathfrak{b}}
\newcommand{\dnumber}{\mathfrak{d}}
\newcommand{\pnumber}{\mathfrak{p}}
\newcommand{\fin}{\mathrm{Fin}}
\renewcommand{\subset}{\subseteq}

\DeclareMathOperator{\ED}{\mathcal{ED}}
\DeclareMathOperator{\add}{add}
\DeclareMathOperator{\non}{non}
\DeclareMathOperator{\cov}{cov}
\DeclareMathOperator{\cof}{cof}
\DeclareMathOperator{\adds}{\add^*}
\DeclareMathOperator{\nons}{\non^*}
\DeclareMathOperator{\covs}{\cov^*}
\DeclareMathOperator{\cofs}{\cof^*}

\DeclareMathOperator{\chf}{\mathbf{1}}

\DeclareMathOperator{\ran}{ran}
\DeclareMathOperator{\FS}{FS}

\begin{document}

\title[The Kat\v{e}tov order in-between $\ED$ and $\fin\otimes\fin$ ]{On the structure of Borel ideals in-between the ideals $\ED$ and $\fin\otimes\fin$ in the Kat\v{e}tov order}

\author{Pratulananda Das}
\address[P.~Das]{Department of Mathematics\\ Jadavpur University\\ Kolkata - 700032\\ West Bengal\\ India}
\email{pratulananda@yahoo.co.in}

\author{Rafa\l{} Filip\'{o}w}
\address[R.~Filip\'{o}w]{Institute of Mathematics\\ Faculty of Mathematics, Physics and Informatics\\ University of Gda\'{n}sk\\ ul. Wita Stwosza 57\\ 80-308 Gda\'{n}sk\\ Poland}
\email{Rafal.Filipow@ug.edu.pl}
\urladdr{http://mat.ug.edu.pl/~rfilipow}

\author{Szymon G\l{}\c{a}b}
\address[Sz.~G\l{}\c{a}b]{Institute of Mathematics\\  \L{}\'{o}d\'{z} University of Technology, W\'{o}lcza\'{n}ska 215, 93-005  \L{}\'{o}d\'{z}, Poland}
\email{szymon.glab@p.lodz.pl}
\urladdr{http://im0.p.lodz.pl/~sglab}

\author{Jacek Tryba}
\address[J.~Tryba]{Institute of Mathematics\\ Faculty of Mathematics, Physics and Informatics\\ University of Gda\'{n}sk\\ ul. Wita Stwosza 57\\ 80-308 Gda\'{n}sk\\ Poland}
\email{Jacek.Tryba@ug.edu.pl}

\date{\today}




\begin{abstract}
For a family $\cF\subseteq \omega^\omega$ we 
define the ideal $\I(\cF)$ on $\omega\times\omega$ 
to be the ideal generated by the family 
$$\{A\subseteq \omega\times\omega:\exists f\in \cF\,\forall^\infty n\, (|\{k:(n,k)\in A\}|\leq f(n))\}.$$
Using ideals of the form $\I(\cF)$, we show that the structure of Borel ideals in-between two well known Borel ideals 
$$\ED = \{A\subseteq\omega\times\omega:\exists m  \, \forall^\infty n\, (|\{k:(n,k)\in A\}|<m))\}$$ 
and 
$$\fin\otimes\fin = \{A\subseteq\omega\times\omega:\forall^\infty n \, (|\{k:(n,k)\in A\}|<\aleph_0))\}$$
 in the Kat\v{e}tov order is fairly complicated. Namely,
there is a copy of $\cP(\omega)/\fin$ 
in-between $\ED$ and $\fin\otimes\fin$, and consequently there are increasing and decreasing chains of length $\bnumber$ and antichains  of size $\continuum$.
\end{abstract}


\maketitle


\setcounter{tocdepth}{1}
\tableofcontents


\section{Introduction}

All the notions and notations used in the introduction are defined in Section~\ref{sec:preliminaries}.

Ideals (and the dual notion -- filters) have traditionally been of much interest in set theory and on the other hand had interesting applications and usages in various
parts of set theory, topology, real analysis and so on (see
e.g.~\cite{MR3545234,MR3712964,MR3436368,MR3247032,MR3490916,MR3436375,MR2520152,MR1711328,MR2777744,MR3692233,MR3351990,MR2491780,MR1708151,MR3453774,MR1758325}
to mention only a few recent publications).

Ideals can play an important role in characterizing other objects.
For instance,
Laczkovich and Rec\l{}aw 
\cite{MR2491780}, and independently 
Debs and Saint Raymond \cite{MR2520152}, used
the ideal
$\fin\otimes\fin$
to characterize the family of all functions of Baire class $1$, whereas
Hru\v{s}\'ak \cite{MR2777744} used the ideal
$\ED$
to characterize selective ultrafilters.
These characterizations are expressed in terms of the Kat\v{e}tov order. Namely, in the first case,  
$\fin\otimes\fin \not\leq_K \I\iff $ $\I$-limits of sequences of continuous functions are of the Baire class one, 
whereas, in the second case, 
$\ED\not\leq_K \cU^*\iff $ $\cU$ is a selective ultrafilter. 

Taking these into account, it should not come as a surprise that a lot of work has been done to examine the structures of ideals in the Kat\v{e}tov order so far (see e.g.~\cite{MR3696069,MR3555332,MR3550610,MR3513296,MR3019575,MR4036731,MR3778962}).

The purpose of this paper is 
to show that the structure of ideals in-between the ideals $\ED$ and $\fin\otimes\fin$  in the Kat\v{e}tov order is quite complicated. Namely, we show that 
there is a copy of $\cP(\omega)/\fin$ 
in-between $\ED$ and $\fin\otimes\fin$, and consequently there are increasing and decreasing chains of length $\bnumber$ and antichains  of size $\continuum$.

In order to obtain these results, we introduced a new class of ideals $\I(\cF)$
(parametrized by $\cF\subseteq \omega^\omega$) 
generated by the family 
$$\{A\subseteq \omega\times\omega:\exists f\in \cF\,\forall^\infty n\, (|\{k:(n,k)\in A\}|\leq f(n))\}.$$
Beside the main results about the Kat\v{e}tov order, we also check other properties of the ideals of the form $\I(\cF)$. For instance, we examine the so called additive properties which can be utilized in the study of ideal convergence of double sequences and we also calculate the values of some well known cardinal characteristics for ideals of the form $\I(\cF)$.


\section{Preliminaries}
\label{sec:preliminaries}

In the sequel we use $\omega$ to denote the set of all natural numbers, and we identify $n\in \omega$ with the set $\{0,1,\dots,n-1\}$ (in particular, $A\setminus n = A\setminus \{0,1,\dots,n-1\}$).

By $\chf_A$ we denote the characteristic function of a set $A$.
In the sequel we assume that $2^\omega=\{0,1\}^\omega$ and $\omega^\omega$ are equipped with the product topology with the discrete topology on $\{0,1\}$ and  $\omega$, respectively.
By identifying sets of natural numbers with their characteristic functions,
we equip $\cP(\omega)$ with the topology of the space $2^\omega$.

For a formula $\phi(x)$ and a set $X$ we write
$\forall^\infty x\in X\,(\phi(x))$ (or simply $\forall^\infty x\,(\phi(x))$)
to abbreviate
that $\phi(x)$ holds for all but finitely many $x\in X$, and
$\exists^\infty x\in X\,(\phi(x))$ (or simply  $\exists^\infty x\in X\,(\phi(x))$) to abbreviate
that $\phi(x)$ holds for infinitely many $x\in X$.
For $f,g\in \omega^\omega$, we write $f\leq^* g$ if $\forall^\infty n \,(f(n)\leq g(n))$. 
For $\cF,\cG\subseteq\omega^\omega$, we  say that $\cF$ is \emph{cofinal} (\emph{$\sigma$-cofinal}, resp.) in $\cG$ if for every $g\in\cG$ (for every $g_0,g_1,\dots\in \cG$, resp.) there is $f\in\cF$ such that $g\leq^* f$ ($g_n\leq^* f$ for each $n$, resp.).
By $\bnumber$ we denote the \emph{bounding number} i.e.~the smallest cardinality of an unbounded subset in $(\omega^\omega,\leq^*)$.
By $\dnumber$ we denote the \emph{dominating number} i.e.~the smallest cardinality of a cofinal subset in $(\omega^\omega,\leq^*)$.
By $\non(\cM)$ we denote the smallest cardinality of non-meager sets in $\omega^\omega$, whereas
$\cov(\cM)$ denotes the smallest cardinality of families $\cF$ of meager sets in $\omega^\omega$ such that $\bigcup\cF=\omega^\omega$.

We write $A\subset^* B$ if $A\setminus B$ is finite.
For a set $A$ and $n\in\omega$,  we define $[A]^n=\{B\subseteq A: |B|=n\}$,
$[A]^{<\omega}=\{B\subseteq A: |B|<\aleph_0\}$
and
$[A]^{\omega}=\{B\subseteq A: |B|=\aleph_0\}$.
For $\cF\subseteq\omega^\omega$, we define 
$\FS(\cF)=\{f_0+\dots+f_n:f_0,\dots,f_n\in\cF\land n\in\omega\}$.
For 
$\cA\subseteq\cP(\omega)$, 
we define
$\cF_\cA = \{f\cdot \chf_{A}:A\in\cA\land f\in\omega^\omega\}$.


\subsection{Ideals}

A  family $\I\subseteq\cP(X)$ is called an \emph{ideal on $X$} if it satisfies the following conditions:
\begin{enumerate}
\item if $A,B\in \I$ then $A\cup B\in\I$,
\item if $A\subseteq B$ and $B\in\I$ then $A\in\I$,
\item $\I$ contains all finite subsets of $X$,
\item $X\notin\I$.
\end{enumerate}
For an ideal $\I$ on $X$, we write $\I^*=\{X\setminus A: A\in\I\}$ and call it the \emph{filter dual to $\I$}.
The ideal of all finite subsets of an infinite set $X$
is denoted by $\fin(X)$ (or $\fin$ for short).

For ideals $\I,\J$ on $X$ and $Y$ respectively, we write 
\begin{enumerate}
\item 
$\I\approx\J$ 
if there is a bijection $f:X\to Y$ such that $A\in\I$ if and only if $f[A]\in\J$ for every $A\subseteq X$ (we say that $\I$ and $\J$ are \emph{isomorphic} in this case);
\item 
$\I\leq_{RK}\J$ 
if there is a function $f:X\to Y$ such that $A\in\I$ if and only if $f[A]\in\J$ for every $A\subseteq X$ ($\leq_{RK}$ is called the \emph{Rudin-Keisler order});
\item 
$\I\leq_K\J$ if there is a function $f:Y\to X$ such that $f^{-1}[A]\in\J$ for every $A\in \I$ ($\leq_K$ is called the \emph{Kat\v{e}tov order});
\item
$\I\leq_{KB}\J$ if there is a finite to one function $f:Y\to X$ such that $f^{-1}[A]\in\J$ for every $A\in \I$ ($\leq_{KB}$ is called the \emph{Kat\v{e}tov-Blass order});
\item
$\I \sqsubseteq \J$ if there is a bijection $f:Y\to X$ such that $f^{-1}[A]\in\J$ for every $A\in \I$ (in this case $\J$ contains an ideal isomorphic to $\I$).
\end{enumerate}

For an ideal $\I$ on $\omega$ and $A\subseteq \omega$ we define $\I\restriction A = \{B\cap A: B\in \I\}$. It is easy to see that $\I\restriction A$ is an ideal on $A$ if and only if $A\notin\I$.

An ideal $\I$ on $X$ is
\begin{enumerate}
\item \emph{tall} if for every infinite $A\subseteq X$ there is an infinite $B\in\I$ such that $B\subseteq A$ (some authors use the name \emph{dense ideal} in this case);
\item \emph{nowhere tall} if $\I\restriction A$ is not tall  for every $A\notin\I$;
\item a \emph{P-ideal} (or \emph{satisfies the condition (AP)}) if for every countable family $\cA\subseteq\I$ there is $B\in\I$ such that $A\setminus B$ is finite for every $A\in\cA$;
\item a \emph{Q-ideal} if for every countable partition $\cF$ of $X$ into finite sets  there is $S\notin \I$ such that $|F\cap S|\leq 1$ for each $F\in \cF$ .
\end{enumerate}

\subsection{Ideals on $\omega\times \omega$}

As this  paper is about ideals on $\omega\times\omega$, before proceeding to introduce our ideals, let us have a quick look into some examples already existing in the literature.

For $A\subseteq \omega\times \omega$ and $n\in \omega$,
we write 
$A_{(n)}=\{k\in \omega: (n,k)\in A\}$
and 
$A^{(n)}=\{k\in \omega: (k,n)\in A\}$
i.e.~$A_{(n)}$ and $A^{(n)}$ are 
the vertical and horizontal (resp.) 
sections of $A$ at the point $n$.

\begin{example}[\cite{MR3016411}]
The ideal 
$$\I_2 = \left\{A\subseteq\omega\times\omega:  \exists k \,\forall^\infty n\, \left(A_{(n)}\cup A^{(n)} \subseteq k\right)\right\}$$
is called the \emph{Pringsheim's ideal}.
\end{example}

A double sequence $x=(x_{m,n})$ of real numbers is said to
be \emph{convergent in Pringsheim's sense} to $L\in \R$ if for any $\varepsilon >0$,
there exists $ {N}\in \omega$ such
that $\left\vert x_{m,n}-L\right\vert <\varepsilon $ whenever
$m,n\geq N$ (\cite{MR1511092}). 
It is easy  to see that the convergence in Pringsheim's sense is equivalent to 
$\I_2$-convergence (i.e.~ideal convergence with respect to the ideal $\I_2$) (\cite{MR3016411}).

An ideal $\I$ on $\omega\times\omega$  \emph{satisfies the condition (AP2)} if for every countable family $\cA \subseteq\I$ there is $B\in \I$ such that 
$A\setminus  B \in \I_2$
for each $A\in \cA$ (\cite{MR2434680}). 

It is known (\cite{MR1844385}) that $\I$-convergence of a double sequence $(x_{n,m})$ can be reduced to the ordinary convergence of a subsequence $(x_{n,m})_{(n,m)\in F}$ with some $F\in \I^*$ if and only if the ideal $\I$ satisfies the condition (AP).
Moreover, one can show (\cite{MR2434680,MR3016411}) that  
$\I$-convergence of a double sequence $(x_{n,m})$ can be reduced to the convergence in Pringsheim's sense of a subsequence $(x_{n,m})_{(n,m)\in F}$ with some $F\in \I^*$ if and only if $\I$ satisfies  the condition (AP2).

\begin{example}[\cite{MR2002719}]
For a set $A\subseteq \omega\times \omega$ 
we write  
$$ \delta_{2}(A) = {\lim_{m,n\rightarrow \infty }}\frac{|A\cap (m\times n)|}{mn}$$
if the considered limit exists in the Pringsheim's sense, and 
then we say that $\delta_{2}(A)$ is the 
\emph{double natural density} of $A$. 
The ideal  
$\I_{\delta_2} = \{A \subseteq \omega\times\omega : \delta_{2}(A)=0\}$
is  called the \emph{ideal of sets of the double natural density zero}.
\end{example}

\begin{example}
\label{exm:FINxFIN-like}
The following families are ideals on $\omega\times\omega$ (see e.g.~\cite{MR2777744}).
\begin{enumerate}
\item
$\{\emptyset\}\otimes\fin=\{A\subseteq\omega\times\omega:\forall n \, (|A_{(n)}|<\aleph_0)\}$.
\item
$\fin\otimes\fin=\{A\subseteq\omega\times\omega:\forall^\infty n \, (|A_{(n)}|<\aleph_0)\}$.
\item
$\fin\otimes\{\emptyset\}=\{A\subseteq\omega\times\omega:\forall^\infty n \, (A_{(n)} = \emptyset)\}$.

\item $\ED=\{A\subseteq\omega\times\omega:\exists m  \, \forall^\infty n \, (|A_{(n)}|<m)\}$.
\end{enumerate}
\end{example}

It is easy to see that all but the last ideals from Example~\ref{exm:FINxFIN-like} can also be defined using the following general notion.
The \emph{Fubini product} of families $\cB,\cC\subseteq\cP(\omega)$ is defined by 
$\cB\otimes\cC =\{A\subseteq\omega\times\omega: \{n\in\omega:A_{(n)}\notin\cC\}\in\cB\}.$


\section{Ideals defined with the aid of infinite sequences of integers}
\label{sec:ideals-generated-by-subfamilies-of-the-Baire-space}

\subsection{Definition and basic properties}

We first introduce our main object for investigation.
\begin{definition}
\label{def:IF-ideal}
For a nonempty family $\cF\subseteq \omega^\omega$ we 
define the ideal $\I(\cF)$ on $\omega\times\omega$ 
to be the ideal generated by the family 
$$\cA = \left\{A\subseteq \omega\times\omega:\exists f\in \cF\, \forall^\infty n \, (|A_{(n)}|\leq f(n))\right\}$$
i.e.~$A\in \I(\cF) \iff A\subseteq A_1\cup\ldots \cup A_n  \text{ for some $A_1,\dots,A_n\in \cA$}.$
\end{definition}

It is easy to see that 
$$\fin\otimes\{\emptyset\}\subseteq\I(\cF)\subseteq \fin\otimes\fin$$
for every nonempty family $\cF$. Moreover, 
some well known ideals on $\omega\times\omega$ are of the form $\I(\cF)$, namely, 
$\fin\otimes\{\emptyset\} = \I(\{\langle 0,0,\dots\rangle \})$,
$\ED = \I(\{f\in\omega^\omega:\text{$f$ is constant}\})$
and
$\fin\otimes\fin = \cI(\omega^\omega)$.

It is easy to see that the Pringsheim's ideal $\I_2$ is not equal to any ideal of the form $\I(\cF)$, though there are some inclusions (for instance $\fin\otimes\{\emptyset\} \subseteq \I_2\subseteq \ED$).

The ideal $\I_{\delta_2}$ of sets of the double natural density zero
is not equal to (even not contained in) any ideal of the form $\I(\cF)$
as it is not difficult to see that  $\I_{\delta_2}$ contains a set having infinite number of infinite vertical sections (for instance $A = \{(m^3,n): m,n\in\omega\} \in \I_{\delta_2}$).
On the other hand, it  happens that ideals of the form $\I(\cF)$ are contained in $\I_{\delta_2}$ (for instance $\fin\otimes\{\emptyset\}\subseteq \I_{\delta_2}$), but not all of them (for instance $\fin\otimes\fin \not\subseteq \I_{\delta_2}$ as $B = \{(m,n): n\leq m\} \in \fin\otimes\fin$ but evidently $\delta_2(B) > 0$).

\begin{proposition}
\label{prop:ideal-on-NtimesN-generated-by-functions}
The following conditions are equivalent.
\begin{enumerate}
\item The family $\cA = \left\{A\subseteq \omega\times\omega:\exists f\in \cF\, \forall^\infty n\, (|A_{(n)}|\leq f(n))\right\}$ is  an ideal on $\omega\times\omega$ (equivalently $\cA = \cI(\cF)$).\label{prop:ideal-on-NtimesN-generated-by-functions-ideal}
\item $\forall f,g\in\cF \, \exists h\in \cF \, (f+g\leq^* h)$.\label{prop:ideal-on-NtimesN-generated-by-functions-bounded}
\end{enumerate}
\end{proposition}

\begin{proof}
$(\ref{prop:ideal-on-NtimesN-generated-by-functions-ideal})\implies(\ref{prop:ideal-on-NtimesN-generated-by-functions-bounded})$
Let $f,g\in\cF$.
Let $A=\{(n,k):k \leq f(n)\}$
and $B=\{(n,k):f(n)< k \leq f(n)+g(n)\}$.
Since $A,B\in\cA$, we get $A\cup B\in\cA$.
Hence there is $h\in\cF$ such that $|(A\cup B)_{(n)}| \leq h(n)$ for all but finitely many $n$.
Since $|(A\cup B)_{(n)}| = |A_{(n)}|+ |B_{(n)}|$ in this case, 
we obtain  $f+g\leq^* h$.

$(\ref{prop:ideal-on-NtimesN-generated-by-functions-bounded})\implies(\ref{prop:ideal-on-NtimesN-generated-by-functions-ideal})$
First note that it is obvious that $\emptyset\in \cA$, $\omega\times\omega\notin\cA$ and $\cA$ contains all finite subsets of $\omega\times\omega$.

Let $A\in \cA$ and $B\subseteq A$.
Let $f\in \cF$ be such that
$|A_{(n)}| \leq f(n)$
for all but finitely many $n$.
Since $|B_{(n)}|\leq |A_{(n)}|$
for all $n$, we have 
$B\in \cA$.

Let $A,B\in \cA$.
Let $f,g\in \cF$ be such that
$|A_{(n)}| \leq f(n)$
and
$|B_{(n)}| \leq g(n)$
for all but finitely many $n$.
Take $h\in \cF$ with $f+g\leq^* h$.
Then   
$|(A\cup B)_{(n)}|\leq   |A_{(n)}|+|B_{(n)}| \leq h(n)$ for all but finitely many   $n$
and $h\in \cF$,
so  
$A\cup B\in \cA$.
\end{proof}

\begin{proposition}
\label{prop:ideal-generated}
$$\I(\cF) = \I(\FS(\cF)) = \left\{A\subseteq \omega\times\omega:\exists f\in \FS(\cF)\, \forall^\infty n \, (|A_{(n)}|\leq f(n))\right\},$$
where $\FS(\cF)=\{f_0+\dots+f_n:f_0,\dots,f_n\in\cF\land n\in\omega\}$.
\end{proposition}

\begin{proof}
The second equality follows from  Proposition~\ref{prop:ideal-on-NtimesN-generated-by-functions}.
The inclusion  $\I(\cF)\subseteq\I(\FS(\cF))$ follows from $\cF\subseteq \FS(\cF)$.
To show $\I(\cF)\supseteq\I(\FS(\cF))$, 
we take $A\in\I(\FS(\cF))$
and, using the second equality,  choose  $f\in \FS(\cF)$ such that
$|A_{(n)}|\leq f(n)$ for all but finitely many $n$.
Let $f_1,\dots,f_m\in\cF$ be such that $f=f_1+\dots+f_m$.
We partition the set $A$ into sets $B_1,\dots, B_m$ such that
$|(B_i)_{(n)}|\leq f_i(n)$ for all but finitely many $n$ and all $i=1,\dots,m$.
Then $B_1,\dots,B_m\in\I(\cF)$, so $A=B_1\cup\dots\cup B_m$ belongs to the ideal $\I(\cF)$.
\end{proof}

The following easy proposition provides a characterization (in terms of $\cF$) showing when some well known ideals on $\omega\times\omega$ are of the form $\I(\cF)$.

\begin{proposition}
\label{prop:characterization-of-known-examples-IF}
\ 
\begin{enumerate}
\item $\I(\cF)=\fin\otimes\fin$ $\iff$ $\FS(\cF)$ is cofinal in $\omega^\omega$.\label{prop:characterization-of-known-examples-IF-FINxFIN}

\item $\I(\cF)=\fin\otimes\{\emptyset\}$ $\iff$ $\forall f\in\cF\, \forall^\infty n\,(f(n)=0)$.\label{prop:characterization-of-known-examples-IF-FINx0}

\item $\ED\subseteq \I(\cF)$ $\iff$
$\exists f\in\FS(\cF)\, \forall^\infty n \, (f(n)\neq0)$.\label{prop:characterization-of-known-examples-IF-ED-subset}

\item $\I(\cF)=\ED$ $\iff$
\begin{enumerate}
\item $\forall f\in\cF\, \exists k  \, \forall^\infty n \, (f(n)\leq k)$,
\item $\exists f\in\FS(\cF)\, \forall^\infty n \, (f(n)\neq0)$.
\end{enumerate}\label{prop:characterization-of-known-examples-IF-ED-equality}
\end{enumerate}
\end{proposition}

\subsection{Additive properties}

\begin{proposition}
\label{prop:IF-is-not-a-P-ideal}
The ideal $\cI(\cF)$ is not a P-ideal (i.e.~does not satisfy the condition (AP)).
\end{proposition}

\begin{proof}
Let $A_n=\{n\}\times\omega$ and note that $A_n\in\I(\cF)$ for every $n$.
Let $B\subseteq\omega\times\omega$ be such that $A_n\setminus B$ is finite for every $n$. Then $B_{(n)}$ is infinite for every $n$, hence $B\notin \I(\cF)$.
\end{proof}

\begin{proposition}
\label{prop:IF-is AP2-ideal}
If $\FS(\cF)$  is $\sigma$-cofinal in $\FS(\cF)$, then 
the ideal $\cI(\cF)$ satisfies the condition (AP2).
\end{proposition}

\begin{proof}
Let $A_0,A_1,\dots\in \I(\cF)$. Let $f_i\in \FS(\cF)$ and $K_i\in \omega$ be such that $|(A_i)_{(n)}|\leq f_i(n)$ for all $n> K_i$.

Since $\FS(\cF)$ is $\sigma$-cofinal in $\FS(\cF)$, there is $f\in \FS(\cF)$ such that for each $i\in \omega$ there is $L_i\in \omega$ with $f_0(n)+\dots+f_i(n)\leq f(n)$ for each $n>L_i$.

Let $M_i\in \omega$ be such that $M_0<M_1<\dots$ and $M_i > \max(K_i,L_i)$ for each $i\in \omega$.

Let  
$$B = \bigcup_{i\in \omega} A_i\setminus (M_i\times \omega).$$

Then $B\in \I(\cF)$ as $|B_{(n)}|\leq f_0(n)+\dots+f_i(n) \leq f(n)$ for each $M_i\leq n<M_{i+1}$ and $i\in \omega$.
Moreover, 
$A_i\setminus B \subseteq  M_i\times \omega\in \I_2$ for each $i\in \omega$.
\end{proof}

The ideal $\fin\otimes\fin = \I(\omega^\omega)$ satisfies the condition (AP2), as the family $\cF=\omega^\omega$ is $\sigma$-cofinal in $\FS(\omega^\omega)$.
On the other hand, 
	$\ED$ is an ideal of the form $\I(\cF)$ that does not satisfy the condition (AP2). Indeed, for each $n\in\omega$ take $A_n=\{(k,nk): k\in\omega \}\in\ED$. Then for every $B\in \ED$, there exist $M\in\omega$ such that for all  but finitely many $k$ we have $|B_{(k)}|\leq M$. Now, one can notice that since $A_n$ are pairwise disjoint, there are at most $M$ sets $A_n$ such that $A_n\setminus B$ is finite, thus there  exists $n$ such that $A_n\setminus B$ is infinite. However, any infinite subset of $\{(k,nk): k\in\omega \}$ does not belong to $\I_2$, hence $A_n\setminus B\not\in\I_2$.

\subsection{Tallness}

\begin{proposition}
\label{prop:IF-is-dense-characterization}
The ideal $\cI(\cF)$ is tall $\iff$
$\forall A\in[\omega]^\omega \, \exists B\in[A]^\omega \, \exists f\in \cF \, \forall n\in B \, (f(n)>0)$.
\end{proposition}

\begin{proof}
$(\implies)$
Let $A\in[\omega]^\omega$.
Since $C=A\times\{0\}$ is infinite, there is an infinite set $D\in\I(\cF)$ such that $D\subseteq C$.
Let $f\in\FS(\cF)$ and $n_0\in\omega$ be such that $|D_{(n)}|\leq f(n)$ for every $n\geq n_0$.
Then $B=\{n\geq n_0:(n,0)\in D\}$ is infinite and $B\subseteq A$.
Moreover, $f(n)\geq |D_{(n)}|=1>0$ for every $n\in B$. Since $f=f_1+\ldots+f_k$ for some  $f_1,\ldots,f_k\in\cF$, there is $i\leq k$ such that $f_i(n)\geq 1$ for  infinitely many $n\in B$.

($\impliedby$)
Let $C\subseteq\omega\times\omega$ be infinite.
If there is $n$ such that $C_{(n)}$ is infinite then $D=\{n\}\times C_{(n)}\in\I(\cF)$ and $D\subseteq C$.
Assume now that $C_{(n)}$ is finite for every $n$.
Since $C$ is infinite, the set $A=\{n: C_{(n)}\neq\emptyset\}$ is infinite.
Let $B\in[A]^\omega$ and $f\in\cF$ be such that $f(n)>0$ for every $n\in B$.
Then for every $n\in B$ there is
$k_n$ such that $(n,k_n)\in C$. Observe that $D=\{(n,k_n):n\in B\}$ is infinite, belongs to $\I(\cF)$ and $D\subseteq C$.
\end{proof}

\begin{corollary}
\label{prop:IF-is-dense-sufficient-condition}
If there exists $f\in\FS(\cF)$ such that $f(n)\neq 0$ for all but finitely many $n$, then the ideal $\cI(\cF)$ is tall.
\end{corollary}

\begin{proposition}
\label{prop:IF-dense-somewhere}
The following conditions are equivalent.
\begin{enumerate}
\item $\I(\cF)$ is nowhere tall (i.e. $\I(\cF)\restriction B$ is not a tall ideal for any $B\notin \I(\cF)$).\label{prop:IF-dense-somewhere:nowhere-tall}

\item $\I(\cF)\restriction (A\times \omega)$ is not a tall ideal for any $A\in[\omega]^\omega$.\label{prop:IF-dense-somewhere:nowhere-tall-on-cylinder}

\item $\I(\cF) = \fin\otimes\{\emptyset\}$.\label{prop:IF-dense-somewhere:equals-FinXempty}
\end{enumerate}
\end{proposition}

\begin{proof}
$(\ref{prop:IF-dense-somewhere:nowhere-tall})\implies (\ref{prop:IF-dense-somewhere:nowhere-tall-on-cylinder})$ 
If $A\in[\omega]^\omega$, then $B=A\times \omega \notin\I(\cF)$. Thus $\I(\cF)\restriction (A\times \omega)$ is not tall.

$(\ref{prop:IF-dense-somewhere:nowhere-tall-on-cylinder})\implies (\ref{prop:IF-dense-somewhere:equals-FinXempty})$
By Proposition~\ref{prop:characterization-of-known-examples-IF}(\ref{prop:characterization-of-known-examples-IF-FINx0}) it is enough to show that for every $f\in \cF$ we have $f(n)=0$ for all but finitely many $n$.
Suppose, to the contrary, that $C=\{n : f(n)\neq 0\}$ is infinite for some $f\in \cF$. 
Then, by Proposition~\ref{prop:IF-is-dense-characterization}, $\I(\cF)\restriction (C\times \omega)$ is tall, a contradiction.

$(\ref{prop:IF-dense-somewhere:equals-FinXempty})\implies (\ref{prop:IF-dense-somewhere:nowhere-tall})$ 
Let $B\notin \I(\cF)=\fin\otimes\{\emptyset\}$. 
Let $A=\{n : B_{(n)}\neq\emptyset\}$ and $k_n \in B_{(n)}$ for every $n\in A$.
Let $C = \{(n,k_n):n\in A\}$.
Then $C\subseteq B$ and $\I(\cF)\restriction   C= \fin$. Hence there is no infinite subset of $C$ that belongs to $\I(\cF)\restriction B$.
\end{proof}


\section{Topological complexity}
\label{sec:topological-complexity}


\subsection{Baire property}

\begin{proposition}
\label{prop:IF-has-BP}
The ideal $\cI(\cF)$ has the Baire property.
\end{proposition}

\begin{proof}
For every $i\in\omega$ we define $F_i=\{(n,k)\in\omega\times\omega: n+k=i\}$. Then $\{F_i:i\in\omega\}$ is a partition of $\omega\times\omega$ into finite sets.
Moreover, it is not difficult to see that if $A\subseteq\omega\times\omega$ contains infinitely many sets $F_i$ then all vertical sections of $A$ are infinite, so $A\notin\I(\cF)$.
Thus, by Talagrand characterization~\cite[Th\'{e}or\`{e}me 21]{MR579439} (see also \cite[Theorem 4.1.2]{MR1350295}), the ideal $\I(\cF)$ has the Baire property.
\end{proof}

Proposition~\ref{prop:IF-has-BP} shows that even if $\cF$ does not have the Baire property in $\omega^\omega$, the ideal $\I(\cF)$ has the Baire property. 
A natural question arises whether there is a set $\cF$ without the Baire property 
such that $\cF$ is cofinal in $\FS(\cF)$.
Below we show  that the answer is positive.

Let $X$ be a topological space.
A set $B\subseteq X$ is a \emph{Bernstein set} if both $B$ and $X\setminus B$ are totally imperfect (i.e.~neither $B$ nor $X\setminus B$ contain a nonempty perfect set).
It is known that Bernstein sets exist in the  space $\omega^\omega$ and they do not have the Baire property (see e.g.~\cite[Theorem~7.20 and 7.22]{MR2778559}).

\begin{proposition}
\label{prop:Bernstein-is-cofinal}
Every  Bernstein set $B\subseteq\omega^\omega$ is cofinal in $\omega^\omega$. In particular, $B$ is cofinal in $\FS(B)$. 
\end{proposition}

\begin{proof}
Let $f\in\omega^\omega$.
Let $A_n=\{f(n)+1,f(n)+2\}$ for every $n$.
Since $P=\prod_{n\in\omega}A_n$ is a nonempty and perfect subset of $\omega^\omega$,
$B\cap P\neq\emptyset$. Moreover, $f\leq g$ for any $g\in B\cap P$.
\end{proof}

Now a natural  question arises whether there is a set $\cF$ without the Baire property such that $\cF$ is cofinal in $\FS(\cF)$ but $\cF$ is not cofinal in $\omega^\omega$.
Below we show that the   answer is positive. 

\begin{lemma}
\label{lem:IFA-continuous-image-of-A}
Let $\phi:\omega^\omega\times \cP(\omega) \to \omega^\omega$ be given by
$\phi(f,A)=(f+1)\cdot \chf_A$.
\begin{enumerate}
\item $\phi[\omega^\omega\times \cA] = \cF_\cA$ (recall that $\cF_\cA = \{f\cdot \chf_{A}:A\in\cA\land f\in\omega^\omega\}$).\label{lem:IFA-continuous-image-of-A:image}

\item $\phi$ is continuous.\label{lem:IFA-continuous-image-of-A:continuous}

\item The inverse images under $\phi$ of  nowhere dense sets are nowhere dense.\label{lem:IFA-continuous-image-of-A:NWD}

\item The inverse images under $\phi$ of   sets with the Baire property have the Baire property.\label{lem:IFA-continuous-image-of-A:BP}
\end{enumerate}
\end{lemma}

\begin{proof}
(\ref{lem:IFA-continuous-image-of-A:image}) Let $f\in\cF_\cA$. Take $A=\omega\setminus f^{-1}[0]$, $g(n)=f(n)-1$ for $n\in A$ and $g(n)=0$ otherwise. Then $A\in\cA$ and $\phi(g,A)=f$. 

(\ref{lem:IFA-continuous-image-of-A:continuous})
Continuity of $\phi$ easily follows from the following simple observation.
If  $(f,A),(g,B)\in \omega^\omega\times \cP(\omega)$ are such that $f\restriction n = g\restriction n$ and $A\cap n = B\cap n$ for some $n\in \omega$, then 
$
\phi(f,A)\restriction n 
= 
\phi(g,B)\restriction n$.

(\ref{lem:IFA-continuous-image-of-A:NWD})
Let $U\subseteq \omega^\omega$ be open and dense.
Then $\phi^{-1}[U]$ is open as $\phi$ is continuous. If we show that $\phi^{-1}[U]$
is also dense, the proof will be  finished.

Let $V = \{(f,A)\in \omega^\omega\times \cP(\omega): s = f\restriction n \land t = A\cap n\}$  with $s\in \omega^n$, $t\in [\omega]^{n}$ and $n\in \omega$ be a basic open set. 
Take any $(f,A)\in V$.
Since $U$ is dense, there is $g\in U\cap \{ h\in \omega^\omega: ((f+1)\cdot \chf_A)\restriction n = h\restriction n  \}$.

Let 
$C = \{i\geq n:g(i)\neq 0\}$
and
$D = \{i\geq n:g(i) = 0\}$.

We define $B = (A\cap n)\cup C$
and
$h = (f\restriction n) \cup ((g-1)\restriction C ) \cup (g\restriction D)$.
Then $(h,B)\in V$ and $\phi(h,B) = (h+1)\cdot\chf_B = g\in U$.
Thus $V\cap \phi^{-1}[U]\neq\emptyset$.

(\ref{lem:IFA-continuous-image-of-A:BP})
 It follows from (2) and (3) as every set with the Baire property is a symmetric difference of an open set and a countable union of nowhere dense sets. 
\end{proof}

\begin{proposition}
\label{prop:ideal-IFA}
Let 
$\cA\subseteq\cP(\omega)$ be an ideal on $\omega$.
\begin{enumerate}
\item 
$\FS(\cF_\cA) = \cF_\cA$. In particular, $\cF_\cA$ is cofinal in  $\FS(\cF_\cA)$.\label{prop:ideal-IFA:cofinal-in-hat}

\item $\cF_\cA$ is not cofinal in $\omega^\omega$.\label{prop:ideal-IFA:not-cofinal-in-all}

\item $\I(\cF_\cA) = (\fin\otimes\fin) \cap (\cA\otimes\{\emptyset\})$. 
\label{prop:ideal-IFA:kwela}

\item If $\cA$ is a nonmeager ideal, then $\cF_\cA$ does not have the Baire property.\label{prop:ideal-IFA:nonmeager}
\end{enumerate} 
\end{proposition}

\begin{proof}
(\ref{prop:ideal-IFA:cofinal-in-hat}) 
easily follows from the fact that $\cA$ is closed under taking finite unions of its elements.
(\ref{prop:ideal-IFA:not-cofinal-in-all})
easily follows from  the fact that $\cA$ does not contain any  cofinite subsets of $\omega$.
(\ref{prop:ideal-IFA:kwela})
is straightforward. (Note that  
ideals defined as an intersection of this form  were   used in \cite{MR3784399}.)

(\ref{prop:ideal-IFA:nonmeager})
Suppose, to the contrary, that $\cF_\cA$ has the Baire property.
Let $\phi:\omega^\omega\times \cP(\omega) \to \omega^\omega$ be given by
$\phi(f,A)=(f+1)\cdot \chf_A$.
Then, by Lemma~\ref{lem:IFA-continuous-image-of-A}, 
$\phi^{-1}[\cF_\cA]$ has the Baire property.
Since $\phi^{-1}[\cF_\cA] = \omega^\omega\times\cA$, by using Kuratowski-Ulam theorem (a topological counterpart of Fubini's theorem) we obtain that $\cA$ has the Baire property. But $\cA$ is a nonmeager ideal, so it does not have the Baire property (see e.g.~\cite[Theorem 4.1.1]{MR1350295}), a contradiction.
\end{proof}

\begin{corollary}
\label{cor:nonBP-cofinal-in-hat-not-in-whole-space}
There is a set $\cF\subseteq\omega^\omega$ without the Baire property such that $\cF$ is cofinal in $\FS(\cF)$, but $\cF$ is not cofinal in $\omega^\omega$.
\end{corollary}


\subsection{Borel complexity}

\begin{proposition}
\label{prop:borel-family-gives-analytic-ideal}
\  
\begin{enumerate}

\item
If $\cF$ is $\sigma$-compact, then
$\cI(\cF)$ is a $\sigma$-compact (hence, $F_\sigma$) ideal.\label{prop:borel-family-gives-analytic-ideal-sigma-compact}

\item If $\cF$ is countable, then $\cI(\cF)$ is an $F_{\sigma}$ ideal.\label{prop:borel-family-gives-analytic-ideal-countable-gives-fsigma}

\item
If $\cF$ is bounded in $(\omega^\omega,\leq^*)$, then
$\cI(\cF)$ is contained in a $\sigma$-compact (hence,  $F_\sigma$)  ideal.\label{prop:borel-family-gives-analytic-ideal-bounded}

\item
If $\cF$ is of cardinality less than $\mathfrak{b}$, then $\I(\cF)$ is  contained in an $F_\sigma$ ideal.\label{prop:borel-family-gives-analytic-ideal-bleow-bnumber-gives-subfsigma}

\item
If $\cF$ is a Borel (or even analytic) set, then
$\cI(\cF)$ is an analytic ideal.\label{prop:borel-family-gives-analytic-ideal-BOREL}
\end{enumerate}
\end{proposition}

\begin{proof}
First observe that 
the set 
$$B=\{(A,f)\in\cP(\omega\times\omega)\times \omega^\omega: \forall^\infty n\in\omega \, |A_{(n)}|\leq f(n)\}$$
is $F_\sigma$, because 
$B = \bigcup_{m\in\omega}\bigcap_{n>m} B_{n}$, where 
 $B_{n}=\{(A,f)\in\cP(\omega\times\omega)\times \omega^\omega:  |A_{(n)}|\leq f(n)\}$
are closed.
Second
observe that 
$\I(\cF)$ is equal to the  the projection of the set $X = (\cP(\omega\times\omega)\times \FS(\cF))\cap B$
onto the first coordinate.

(\ref{prop:borel-family-gives-analytic-ideal-sigma-compact})
If $\cF$ is $\sigma$-compact, then $X$ is $\sigma$-compact (as the intersection of the $\sigma$-compact set $\cP(\omega\times\omega)\times \FS(\cF)$ and the $F_\sigma$ set $B$). 
Consequently $\I(\cF)$ is $\sigma$-compact as a continuous image of $X$.

(\ref{prop:borel-family-gives-analytic-ideal-countable-gives-fsigma}) 
It follows from (\ref{prop:borel-family-gives-analytic-ideal-sigma-compact}) because every countable family is $\sigma$-compact.

(\ref{prop:borel-family-gives-analytic-ideal-bounded})
It follows from (\ref{prop:borel-family-gives-analytic-ideal-sigma-compact}) as every bounded set in $(\omega^\omega,\leq^*)$ is contained in a $\sigma$-compact set.

(\ref{prop:borel-family-gives-analytic-ideal-bleow-bnumber-gives-subfsigma})
It follows from (\ref{prop:borel-family-gives-analytic-ideal-bounded}) because every family of cardinality less than $\bnumber$ is bounded in $(\omega^\omega,\leq^*)$.

(\ref{prop:borel-family-gives-analytic-ideal-BOREL})
If $\cF$ is analytic, then $X$ is analytic as the intersection of the analytic set $\cP(\omega\times\omega)\times \FS(\cF)$ 
and the $F_\sigma$ set $B$.
\end{proof}

\begin{lemma}
\label{lem:IF-is-ideal-continuous-preimage-of-IF}
Let $\Phi:\cP(\omega)\to\cP(\omega\times\omega)$
be given by $\Phi(A) = A\times\{0\}$.
\begin{enumerate}
    \item $\Phi$ is continuous.
\item 
If $\cA\subseteq\cP(\omega)$ is an ideal on $\omega$, then 
$\Phi^{-1}[\I(\cF_\cA)] =\cA$.
\end{enumerate}
\end{lemma}

\begin{proof}
Since continuity of $\Phi$ is straightforward,  we only verify (2).

$(\subseteq)$
Let $B\in \Phi^{-1}[\I(\cF_\cA)]$.
Then $\Phi(B)\in \I(\cF_\cA)$, so there is $A\in \cA$ and $f\in \omega^\omega$ such that 
$|(\Phi(B))_{(n)}|\leq f\cdot \chf_A(n)$ for all but finitely many $n$.
On the other hand,
$|(\Phi(B))_{(n)}| = |(B\times\{0\})_{(n)}|=\chf_B(n)$ for every $n$.
Thus, $\chf_B(n)\leq \chf_A(n)$ for all but finitely many $n$. 
Hence $B \subset^* A$.
Since $A\in \cA$ and $\cA$ is an ideal, $B\in \cA$. 

$(\supseteq)$
Let $A\in\cA$. Then $|(\Phi(A))_{(n)}|=|(A\times\{0\})_{(n)}| = \chf_A(n)$
for every $n$ and $\chf_A\in \cF_\cA$. Thus $\Phi(A)\in \I(\cF_\cA)$.
\end{proof}

\begin{proposition}
\label{prop:ideal-IFA-borel-or-not}
Let 
$\cA\subseteq\cP(\omega)$ be an ideal on $\omega$.
\begin{enumerate}
\item 
If $\cA $ is a Borel  ideal, then $\I(\cF_\cA)$ is a Borel ideal.\label{prop:ideal-IFA-borel-or-not:Borel}

\item 
There are Borel ideals of the form $\I(\cF)$ of arbitrarily high Borel complexity i.e.~for every $\alpha<\omega_1$ there exists $\cF\subseteq\omega^\omega$ such that the ideal $\I(\cF)$ is Borel but not in $\bSigma^0_\alpha$.\label{prop:ideal-IFA-borel-or-not:Borel-not-Sigma-alpha}

\item 
There exists $\cF\subseteq\omega^\omega$ such that the ideal $\I(\cF)$ is not Borel.\label{prop:ideal-IFA-borel-or-not:not-Borel}

\end{enumerate}
\end{proposition}

\begin{proof}
(\ref{prop:ideal-IFA-borel-or-not:Borel})
Since
$\cA\in\bSigma^0_\alpha$ for some $\alpha<\omega_1$
and  $\{\emptyset\} \in \bPi^0_1$, we obtain $\cA\otimes\{\emptyset\}\in \bSigma^0_{1+\alpha}$ 
(see e.g.~ \cite[Proposition~1.6.16]{alcantara-phd-thesis}). 
Since $\fin\otimes\fin\in \bSigma^0_4$, we have 
$\I(\cF_\cA)=(\fin\otimes\fin) \cap (\cA\otimes\{\emptyset\}) \in \bSigma^0_{\beta}$, where $\beta=\max\{4,1+\alpha\}$. All in all, $\I(\cF_\cA)$ is a Borel ideal.

(\ref{prop:ideal-IFA-borel-or-not:Borel-not-Sigma-alpha})
Let $\alpha<\omega_1$.
Let $\cA$ be a Borel ideal on $\omega$ such that $\cA\notin\bSigma^0_\alpha$ (for the existence of such ideals see e.g.~\cite{MR817590}, \cite{MR1004056} or \cite{MR1321463}).
By (\ref{prop:ideal-IFA-borel-or-not:Borel}), $\I(\cF_\cA)$ is a Borel ideal, but Lemma~\ref{lem:IF-is-ideal-continuous-preimage-of-IF} implies that $\I(\cF_\cA) \notin \bSigma^0_\alpha$.

(\ref{prop:ideal-IFA-borel-or-not:not-Borel})
If  $\cA\subseteq\cP(\omega)$ is an ideal which is not Borel (for instance a maximal ideal), then 
$\I(\cF_\cA)$ is not Borel by  Lemma~\ref{lem:IF-is-ideal-continuous-preimage-of-IF}.
\end{proof}


\section{How many ideals are there?}
\label{sec:how-many}

\begin{lemma}
\label{lem:K-KB-orders-relations-between-A-and-IFA}
Let $\cA$ and $\cB$ be ideals on $\omega$.
\begin{enumerate}
    \item 
If $\I(\cF_\cB) \leq_K\I(\cF_\cA)$, then $\cB\leq_{KB}\cA$.\label{lem:K-KB-orders-relations-between-A-and-IFA:1}

\item 
If $\cB\leq_{KB}\cA$, then  $\I(\cF_\cB)\leq_{KB}\I(\cF_\cA)$.\label{lem:K-KB-orders-relations-between-A-and-IFA:2}


\item 
If $\cB\leq_{K}\cA$ and $\cA$ is a P-ideal, then  $\I(\cF_\cB)\leq_{KB}\I(\cF_\cA)$.\label{lem:K-KB-orders-relations-between-A-and-IFA:4}

    \item 
$\I(\cF_\cB)\leq_{K}\I(\cF_\cA)$ if and only if $\I(\cF_\cB)\leq_{KB}\I(\cF_\cA)$.\label{lem:K-KB-orders-relations-between-A-and-IFA:3}

\end{enumerate}
\end{lemma}

\begin{proof}
(\ref{lem:K-KB-orders-relations-between-A-and-IFA:1})
Let  $\Phi:\omega\times\omega\rightarrow \omega\times\omega$ be a function such that $\Phi^{-1}[B]\in\I(\cF_\cA)$
for every $B\in\I(\cF_\cB)$.
Below we define a function $\Psi:\omega\to\omega$ which will witness $\cB\leq_{KB}\cA$.

Let $i_0=0$ and define $m_0=\min\{m\in\omega: \Phi[\{i_0\}\times\omega]\cap (\{m\}\times\omega)\not=\emptyset \}$.  
Since $\Phi^{-1}[\{m_0\}\times\omega]\in \I(\cF_\cA)$, 
the set $F_0=\{i_0\} \cup \{i\in\omega: \Phi[\{i\}\times\omega]\subseteq\{m_0\}\times\omega\}$ is finite. 
For every $i\in F_0$, 
we define $\Psi(i)=m_0$.

We proceed inductively. 
Let 
$i_n = \min\left(\omega\setminus  \bigcup_{k<n}F_k\right)$
and define
$$m_n=\min\left\{m\in\omega\setminus\{m_k:k<n\}: \Phi\left[\left\{i_n\right\} \times\omega\right]\cap (\{m\}\times\omega)\neq \emptyset \right\}.$$ 
Since $\Phi^{-1}[\{G\}\times\omega]\in \I(\cF_\cA)$ for every finite set $G\subseteq \omega$, the set  
$$F_n=
\{i_n\} \cup 
\left\{i\in\omega\setminus \bigcup_{k<n}F_k: 
\Phi[\{i\}\times\omega]
\subseteq
 \{m_k:k\leq n\}\times\omega\right\}$$
 is finite. 
For every $i\in F_n$, 
we define $\Psi(i)=m_n$.

Clearly, $\Psi$ is a finite-to-one function. Once we show that $\Psi^{-1}[B]\in\cA$ for every $B\in\cB$, the proof will be finished.

Take $B\in\cB$. 
For every $i\in F_n$, $n\in\omega$, we pick   $(m_n,c_i)\in \Phi[\{i\}\times\omega]\cap (\{m_n\}\times\omega)$.
Let $C=\{(m_n,c_i): n\in \omega,i\in F_n\} \cap (B\times\omega)$. Since $C_{(m)}$ is finite for every $m\in B$ and empty for each $m\in \omega\setminus B$, there exists a function $f\in\omega^\omega$ such that $|C_{(m)}|\leq (f\cdot \chf_B)(m)$ for every $m\in\omega$. Thus, $C\in\I(\cF_\cB)$, so $\Phi^{-1}[C]\in\I(\cF_\cA)$.
 It follows that there is a function $g\in\omega^\omega$ and a set $A\in\cA$ such that $|(\Phi^{-1}[C])_{(i)}|\leq (g\cdot \chf_A)(i)$ for all but finitely many $i\in\omega$. 
Then
it is not difficult to see that 
$$\Psi^{-1}[B] = \bigcup_{m_n\in B} F_n
=\{i\in\omega: (\Phi^{-1}[C])_{(i)}\neq\emptyset\}
\subseteq^* A,$$ 
so $\Psi^{-1}[B]\in \cA$. 

(\ref{lem:K-KB-orders-relations-between-A-and-IFA:2})
Let 
$\Psi:\omega\rightarrow \omega$ 
be a finite-to-one function 
such that $\Psi^{-1}[B]\in\cA$
for every $B\in\cB$.
Below we define a function 
$\Phi:\omega\times\omega\rightarrow \omega\times\omega$ 
which will witness $\I(\cF_\cB)\leq_{KB}\I(\cF_\cA)$.

Let $\ran(\Psi) = \{m_n:n\in\omega\}$ and $F_n=\Psi^{-1}[\{m_n\}]$ for $n\in \omega$. 

Let $\Phi:\omega\times\omega\to\omega\times\omega$ be a one-to-one function such that $\Phi[F_n\times\omega] = \{m_n\}\times\omega$ for each $n\in \omega$.

Let $C\in \I(\cF_\cB)$, and take  $B\in \cB$, $f\in \omega^\omega$ and $M\in \omega$  such that $|(C)_{(m)}|\leq (f\cdot \chf_B)(m)$ for all $m\geq M$.

Let   
$A = \bigcup\{ F_n:m_n\in B\setminus M\} = \Psi^{-1}[B\setminus M]\in \cA$
and
$F = \bigcup\{ F_n: m_n\in B\cap M\} \in\fin$.

Since 
$(\Phi^{-1}[C])_{(i)}$ is finite for all $i\in A$ and 
$(\Phi^{-1}[C])_{(i)} = \emptyset$  for all $i\in \omega\setminus(A\cup F)$, 
there is $g\in \omega^\omega$ with 
$|(\Phi^{-1}[C])_{(i)}|\leq (g\cdot \chf_A)$ for all $i\in \omega\setminus F$, and consequently  $\Phi^{-1}[C]\in \I(\cF_\cA)$.

(\ref{lem:K-KB-orders-relations-between-A-and-IFA:4})
It follows from (\ref{lem:K-KB-orders-relations-between-A-and-IFA:2}) and the fact that if $\cA$ is a P-ideal, then $\cB\leq_{K}\cA$ implies $\cB\leq_{KB}\cA$.

(\ref{lem:K-KB-orders-relations-between-A-and-IFA:3})
It follows from (\ref{lem:K-KB-orders-relations-between-A-and-IFA:1})
and
(\ref{lem:K-KB-orders-relations-between-A-and-IFA:2}).
\end{proof}

From the above lemma we easily obtain  the following theorem which says, in particular, that the structure of the Kat\v{e}tov-Blass order among the ideals of the form $\I(\cF)$ is as complicated as the structure of the Kat\v{e}tov-Blass order among all ideals on $\omega$.

\begin{theorem}
\label{thm:embedding-all-ideals-into-our-ideals}
If $\cA$ and $\cB$ are ideals on $\omega$, then 
$$\cB\leq_{KB}\cA \iff \I(\cF_\cB)\leq_{KB}\I(\cF_\cA)$$
i.e.~there is an order embedding of the family of all ideals, ordered by the Kat\v{e}tov-Blass order, into the family of ideals of the form $\I(\cF)$,  ordered  by the Kat\v{e}tov-Blass order.
\end{theorem}

Using the above theorem we prove the following theorem which says, among others, that there are as many as possible pairwise nonisomorphic ideals of the form $\I(\cF)$.

\begin{theorem}\ 
\label{thm:2-to-continuum-nonisomorphic}
There are $2^\continuum$ pairwise $\leq_{K}$-incomparable ideals of the form $\I(\cF)$ (i.e.~there is $\leq_{K}$-antichain of cardinality $2^\continuum$ of ideals of the form $\I(\cF)$). In particular,
there are $2^\continuum$ pairwise nonisomorphic ideals of the form $\I(\cF)$. 
\end{theorem}

\begin{proof}
First, we notice that if $\cA$ and $\cB$ are $\leq_{RK}$-incomparable \emph{maximal} ideals, then $\I(\cF_\cA)$ and $\I(\cF_\cB)$ are $\leq_K$-incomparable. Indeed, suppose that  
$\I(\cF_\cA)$ and $\I(\cF_\cB)$ are $\leq_K$-comparable i.e.~$\I(\cF_\cA)\leq_K\I(\cF_\cB)$ or $\I(\cF_\cB)\leq_K\I(\cF_\cA)$.
Then, by Lemma~\ref{lem:K-KB-orders-relations-between-A-and-IFA}, $\cA\leq_{KB} \cB$ or $\cB\leq_{KB}\cA$. 
Now, using maximality of $\cA$ and $\cB$, it is easy to see that $\cA\leq_{RK}\cB$ or $\cB\leq_{RK}\cA$, so $\cA$ and $\cB$ are $\leq_{RK}$-comparable.

Second, we take  $2^\continuum$ pairwise $\leq_{RK}$-incomparable maximal ideals (see \cite{MR540490}) and  obtain  $2^\continuum$ pairwise $\leq_K$-incomparable ideals of the form $\I(\cF_\cA)$.
\end{proof}

\subsection{Borel ideals}

\begin{lemma}
\label{lem:Q-ideals}
If  $\cA$ is  an ideal on $\omega$ which is \emph{not} a  Q-ideal, then  
$\ED\sqsubseteq \I(\cF_\cA)$ (in particular, $\ED\leq_{KB}\I(\cF_\cA)$).
\end{lemma}

\begin{proof}
Since $\cA$ is not a Q-ideal, there is a partition $\{F_n:n\in \omega\}$ of $\omega$ into finite sets such that for every $S\notin\cA$
there is $n\in \omega$ with $|F_n\cap S|\geq 2$.

Let $G:\omega\times\omega\to\omega\times\omega$ be a bijection such that $G[F_n\times \omega]=\{n\}\times\omega$ for each $n$.

We claim that $G$ witnesses $\ED\sqsubseteq  \I(\cF_\cA)$ i.e.~$G^{-1}[B]\in \I(\cF_\cA)$ for every $B\in \ED$.

Since the ideal $\ED$ is generated by vertical lines $\{n\}\times\omega$ and functions $f\in \omega^\omega$ it is enough to check that 
$G^{-1}[\{n\}\times\omega] \in \I(\cF_\cA)$ and  $G^{-1}[f]\in \I(\cF_\cA)$
for every $n\in \omega$ and $f\in \omega^\omega$.

For any $n\in \omega$, we have $G^{-1}[\{n\}\times\omega] = F_n\times\omega\in \I(\cF_\cA)$.

Take any  $f\in \omega^\omega$.
Let $A$ be the projection of the set $G^{-1}[f]$ onto the first coordinate.
Then it is not difficult to see that $|A\cap F_n|=1$ for each $n\in \omega$, so $A\in \cA$.
Since $|(G^{-1}[f])_{(n)}|\leq \chf_A(n)$ for each $n$, we get  $G^{-1}[f]\in \I(\cF_\cA)$
\end{proof}

\begin{theorem}\ 
\label{thm:continuum-nonisomorphic-Borel}
There is an order embedding of $\cP(\omega)/\fin$, ordered by $\subseteq^*$, into 
the family of Borel (in fact $\bSigma^0_4$) ideals of the form $\I(\cF)$
which are in-between  the ideals $\ED$ and $\fin\otimes\fin$, 
ordered by the Kat\v{e}tov (or equivalently Kat\v{e}tov-Blass) order.
In particular,
\begin{enumerate}
    \item 
there is a $\leq_{K}$-antichain of cardinality $\continuum$ of Borel ideals of the form $\I(\cF)$ (in particular, there are $\continuum$ pairwise nonisomorphic Borel ideals of the form $\I(\cF)$);\label{thm:continuum-nonisomorphic-Borel:1}

\item 
there are increasing and decreasing $\leq_{KB}$-chains of length $\bnumber$ of Borel ideals of the form $\I(\cF)$.\label{thm:continuum-nonisomorphic-Borel:2}
\end{enumerate}
\end{theorem}

\begin{proof}
In \cite{MR3513296,MR3550610}, the authors proved that there is 
 an order embedding of $\cP(\omega)/\fin$, ordered by $\subseteq^*$, into 
the family of tall $F_\sigma$ P-ideals, ordered by the Kat\v{e}tov order. 
 Since  $\I\leq_K\J \iff \I\leq_{KB}\J$  for P-ideals, we 
see that there is also  
 an order embedding of $\cP(\omega)/\fin$, ordered by $\subseteq^*$, into 
the family of tall $F_\sigma$ P-ideals, ordered by the Kat\v{e}tov-Blass  order. 
Now using Theorem~\ref{thm:embedding-all-ideals-into-our-ideals}, we obtain 
an order embedding of $\cP(\omega)/\fin$, ordered by $\subseteq^*$, into 
the family of ideals of the form $\I(\cF)$, ordered by the Kat\v{e}tov-Blass order.
Since the Kat\v{e}tov order is equivalent to the Kat\v{e}tov-Blass order in the realm of ideals of the form $\I(\cF)$ (see Lemma~\ref{lem:K-KB-orders-relations-between-A-and-IFA}(\ref{lem:K-KB-orders-relations-between-A-and-IFA:3})), we also obtain an order embedding of $\cP(\omega)/\fin$, ordered by $\subseteq^*$, into 
the family of ideals of the form $\I(\cF)$, ordered by the Kat\v{e}tov order.

If $\cA$ is an $F_\sigma$-ideal, then by  Proposition~\ref{prop:ideal-IFA-borel-or-not} the ideal $\I(\cF_\cA)$ is  Borel (from the proof of Proposition~\ref{prop:ideal-IFA-borel-or-not} it follows that  the ideal $\I(\cF_\cA)$ is  in fact $\bSigma^0_4$).

All ideals of the form $\I(\cF)$ are contained in $\fin\otimes\fin$, so they are $\leq_{KB}$-below $\fin\otimes\fin$.

In \cite{MR2905404}, the authors proved that if $\cA$ is a tall $F_\sigma$-ideal, then $\cA$ is not a Q-ideal, so  
Lemma~\ref{lem:Q-ideals} gives $\ED\leq_{KB}\I(\cF_\cA)$.

(\ref{thm:continuum-nonisomorphic-Borel:1})
It follows from the fact that there are $\subseteq^*$-antichains of cardinality $\continuum$ in $\cP(\omega)/\fin$.

(\ref{thm:continuum-nonisomorphic-Borel:2})
It follows from the fact that there are increasing and decreasing $\subseteq^*$-chains of length $\bnumber$ in  $\cP(\omega)/\fin$.
\end{proof}

\subsection{$F_\sigma$-ideals}

For  $f\in \omega^\omega$ we write $\cF_f=\{k\cdot f:k\in\omega\}$.
By Proposition~\ref{prop:borel-family-gives-analytic-ideal}(\ref{prop:borel-family-gives-analytic-ideal-countable-gives-fsigma}), the ideal $\I(\cF_f)$ is $F_\sigma$ for any $f\in \omega^\omega$.

\begin{lemma}
\label{lem:Fsigma-isomorphism}
Let $f,g\in \omega^\omega$.
If $g\neq^*0$
and
$$\forall M\,\exists N\,\forall n,k>N\,\left(\frac{f(n)}{g(k)}>M \lor \frac{g(k)}{f(n)}>M\right),$$
then $\I(\cF_g)$ and $\I(\cF_f)$ are not isomorphic.
\end{lemma}

\begin{proof}
Suppose for the sake of contradiction that 
$\I(\cF_g)$ and $\I(\cF_f)$ are isomorphic.
Let $\Phi:\omega\times\omega\to\omega\times\omega$ 
be a bijection
such that $B\in \I(\cF_g) \iff \Phi^{-1}[B]\in \I(\cF_f)$.

\begin{claim}
$\Phi[\{i\}\times\omega] \in \fin\otimes\{\emptyset\}$ for all but finitely many $i\in \omega$.
\end{claim}

\begin{proof}[Proof of Claim]
Suppose for the sake of contradiction that 
$\Phi[\{i\}\times\omega] \notin \fin\otimes\{\emptyset\}$ for infinitely  many $i$.
Let $i_1<i_2<\dots$ be such that 
$\Phi[\{i_n\}\times\omega] \notin \fin\otimes\{\emptyset\}$ for each $n$.
Now, we  inductively pick, using the diagonal argument, 
pairwise distinct elements $a^n_k$ 
and some $b^n_k$ for $n,k\in \omega$ such that 
$(a^n_k,b^n_k)\in \Phi[\{i_n\}\times\omega]$.

If $B = \{(a^n_k,b^n_k):n,k\in\omega\}$,
then $|B_{(i)}|\leq 1 \leq g(i)$ for all but finitely many $i$.
Thus $B\in \I(\cF_g)$, and consequently $\Phi^{-1}[B]\in \I(\cF_f)$.

On the other hand, 
$\Phi^{-1}[\{(a^n_k,b^n_k):k\in\omega\}]\subseteq\{i_n\}\times\omega$
for each $n$, so $(\Phi^{-1}[B])_{(i_n)}$ is infinite for each $n$.
This means that $\Phi^{-1}[B]\notin\I(\cF_f)$, a contradiction.
\end{proof}

\begin{claim}
There are $a_0<a_1<\dots$ and $b_0<b_1<\dots$  such that 
$(\Phi[\{a_n\}\times\omega])_{(b_n)}$ is infinite for each $n\in\omega$.
\end{claim}

\begin{proof}[Proof of Claim]
By the previous Claim, there is $a_0$ such that $\Phi[\{a\}\times\omega] \in \fin\otimes\{\emptyset\}$ for each $a\geq a_0$.
Take $b_0$ such that $(\Phi[\{a_0\}\times\omega])_{(b_0)}$ is infinite.

There is $a_1>a_0$ such that $\Phi[\{a\}\times\omega]\setminus \{0,1,\dots,b_0\}\times \omega$ is infinite for each $a\geq a_1$
(otherwise the set   
$(\Phi^{-1}[\{0,1,\dots,b_0\}\times \omega])_{(a)}$ would be cofinite for infinitely many $a\geq a_0$, 
and consequently 
$\Phi^{-1}[\{0,1,\dots,b_0\}\times \omega]\notin \I(\cF_f)$, a contradiction).
Take $b_1>b_0$ such that $(\Phi[\{a_1\}\times\omega])_{(b_1)}$ is infinite.

It is not difficult to see that proceeding by induction we obtain the required sequences $(a_n)$ and $(b_n)$
\end{proof}

Now, we are ready to finish the proof of the lemma.
Let $(a_n)$ and $(b_n)$ be as in the last Claim, and take  a sequence $k_1<k_2<\dots$ such that either $$\lim_{n\to\infty}\frac{f(a_{k_n})}{g(b_{k_n})}=\infty
\text{\ \ or\ \ }
\lim_{n\to\infty}\frac{g(b_{k_n})}{f(a_{k_n})}=\infty.$$

In the former case, 
we  pick a set $A\subseteq \{a_{k_n}:n\in \omega\}\times\omega$ such that
$|A_{(a_{k_n})}|= |(\Phi[A])_{(b_{k_n})}| = f(a_{k_n})$ for each $n$.
Then $A\in \I(\cF_f)$ and consequently 
$\Phi[A]\in\I(\cF_g)$, 
so there is $k$ with $|(\Phi[A])_{(i)}|\leq k\cdot g(i)$ for all but finitely many $i$.
It means that $f(a_{k_n})\leq k\cdot g(b_{k_n})$
for all but finitely many $n$, a contradiction with 
$\lim_{n} f(a_{k_n})/g(b_{k_n})=\infty$.

In the latter case,  we  pick a set $B\subseteq \{b_{k_n}:n\in \omega\}\times\omega$ such that
$|B_{(b_{k_n})}|= |(\Phi^{-1}[B])_{(a_{k_n})}|= g(b_{k_n})$ for each $n$.
Then $B\in \I(\cF_g)$ and consequently $\Phi^{-1}[B]\in\I(\cF_f)$, so there is $k$ with $|(\Phi^{-1}[B])_{(i)}|\leq k\cdot f(i)$ for all but finitely many $i$.
It means that $g(b_{k_n})\leq k\cdot f(a_{k_n})$
for all but finitely many $n$, a contradiction with 
$\lim_{n} g(b_{k_n})/f(a_{k_n})=\infty$.
\end{proof}

\begin{theorem}
There are $\continuum$ pairwise nonisomorphic $F_\sigma$ ideals of the form $\I(\cF)$.
\end{theorem}

\begin{proof}
Let $\cA$ be a family of $\continuum$ pairwise almost disjoint infinite subsets of $\omega$.
Let $c^A_0<c^A_1<\dots$ be the increasing enumeration of $A\in \cA$.

For each $A\in\cA$ we define  $f_A:\omega\to\omega$ by 
$f_A(n) = (c^A_n)!$ for each $n$.

If we show that for distinct $A,B\in \cA$ the functions $f=f_A$ and $g=f_B$ satisfy the assumption from Lemma~\ref{lem:Fsigma-isomorphism}, the proof of the theorem will be finished.

Let $M$ be fixed.
Since $A\cap B$ is finite, there is $N>M$ such that $c^A_n\neq c^B_k$ for each $n,k>N$.
Taking $n,k>N$ we have either $c^A_n<c^B_k$ or $c^A_n>c^B_k$, so 
either 
$$\frac{f_B(k)}{f_A(n)} = \frac{(c^B_k)!}{(c^A_n)!} \geq c^B_k \geq k>N>M
\text{\ \ or\ \ }
\frac{f_A(n)}{f_B(k)} = \frac{(c^A_n)!}{(c^B_k)!} \geq c^A_n \geq n>N>M.$$
\end{proof}

\begin{question}
Is there an order embedding of $\cP(\omega)/\fin$, ordered by $\subseteq^*$, into 
the family of $F_\sigma$ ideals of the form $\I(\cF)$, ordered by the Kat\v{e}tov order?
Are there  uncountable increasing (decreasing, resp.)  $\leq_{K}$-chains of  $F_\sigma$ ideals of the form $\I(\cF)$?
Are there  uncountable  $\leq_{K}$-antichains of  $F_\sigma$ ideals of the form $\I(\cF)$?
\end{question}


\section{Cardinal characteristics of ideals}
\label{sec:cardinal-characteristics}

Some properties of ideals can be described by cardinal characteristics associated with them. There are four well known cardinal characteristics called additivity, covering, uniformity and cofinality defined in the following way
for an ideal $\I$ on $X$ (see e.g.~\cite{MR1350295}):
$\add(\I) =\min\left\{|\cA|:\cA\subseteq\I\land \bigcup\cA\notin\I\right\}$,
$\cov(\I)  = \min\left\{|\cA|: \cA\subseteq\I\land \bigcup\cA=X\right\}$, 
$\non(\I)  = \min\{|A|:A\notin \I\}$, 
$\cof(\I)  =\min\{|\cA|:\cA\subseteq\I\land \forall B\in\I \, \exists A\in\cA \, (B\subseteq A)\}$.
These  characteristics are useful in the case of ideals on an uncountable set $X$ (for instance in the case of the $\sigma$-ideal  of all meager sets and the $\sigma$-ideal of all Lebesgue null sets). However, they are (but $\cof$) useless in the case of ideals on $\omega$, since $\add(\I)=\cov(\I)=\non(\I)=\aleph_0$ for every ideal $\I$ on $\omega$.
Fortunately, 
Hern\'{a}ndez and Hru\v{s}\'{a}k
introduced in \cite{MR2319159} (see also~\cite{MR2777744}) certain versions of these characteristics more suitable for tall ideals on $\omega$.
Namely, for each tall ideal $\I$  they define:
\begin{equation*}
\begin{split}
\adds(\I) & =\min\{|\cA|:\cA\subseteq\I\land \neg \exists B\in\I \, \forall A\in\cA \, (A\subset^* B)\}, \\
\covs(\I)  & = \min\{|\cA|: \cA\subseteq\I\land \neg \exists B\in\cP(\omega)\setminus\fin^* \, \forall A\in\cA \, (A\subset^* B)\}, \\
\nons(\I) & = \min\{|\cA|:\cA\subseteq \cP(\omega)\setminus\fin^*\land \forall B\in\I \, \exists A\in\cA \, (B\subset^* A)\}, \\
\cofs(\I) & =\min\{|\cA|:\cA\subseteq\I\land \forall B\in\I \, \exists A\in\cA \, (B\subset^* A)\}.
\end{split}
\end{equation*}

If $\cU$ is a free ultrafilter on $\omega$ and $\I=\cU^*$ is the dual ideal, the above mentioned characteristics were earlier introduced by Brendle and Shelah in \cite{MR1686797}  were the authors used the notations 
$\pnumber(\I)$, $\pi\pnumber(\I)$, $\pi\chi(\I)$ and $\chi(\I)$ for 
$\adds(\I)$, $\covs(\I)$, $\nons(\I)$ and $\cofs(\I)$ respectively.

If an ideal $\I$ is not tall then $\covs(\I)$ is not well-defined (as the $\min$ would be taken from the empty set) and $\nons(\I)=1$.

There are some inequalities holding among these characteristics for all tall ideals (see e.g.~\cite[p.~578]{MR2777744}), namely:
$\aleph_0\leq \adds(\I)\leq \covs(\I)\leq \cofs(\I)\leq \continuum$
and
$\aleph_0\leq \adds(\I)\leq \nons(\I)\leq \cofs(\I)\leq \continuum$.

It is easy to check (see \cite[Theorems~1.6.4 and 1.6.19]{alcantara-phd-thesis}) that 
$\adds(\ED)=\adds(\fin\otimes\fin)=\aleph_0$, 
$\nons(\ED)=\nons(\fin\otimes\fin)=\aleph_0$.
Using almost the same argument one can easily show the following.

\begin{proposition}
\label{prop:add-equals-omega}
$\adds(\I(\cF))=\aleph_0$ and 
$\nons(\I(\cF))=\aleph_0$
for all tall ideals $\I(\cF)$.
\end{proposition}


\subsection{Cofinality}

It is known that 
$\cofs(\fin\otimes\fin)=\dnumber$ (see \cite[Theorem~1.6.19]{alcantara-phd-thesis}).
Below we show that $\dnumber$ is a lower bound of the cardinal $\cofs$ for every ideal of the form $\I(\cF)$ but $\fin\otimes\{\emptyset\}$.

\begin{proposition}
\label{prop:cofs-geq-d}
$\cofs(\I(\cF))\geq \dnumber$ if and only if  $\I(\cF)\neq\fin\otimes\{\emptyset\}$.
\end{proposition}

\begin{proof}
($\implies$) 
It follows from the easy fact that $\cofs(\fin\otimes\{\emptyset\}) = \aleph_0< \dnumber$.

($\impliedby$)
By Proposition~\ref{prop:characterization-of-known-examples-IF}(\ref{prop:characterization-of-known-examples-IF-FINx0})
there is $f\in\cF$ such that the set $A=\{n\in\omega:f(n)\neq0\}$ is infinite.
Suppose, to the contrary, that there is $\cB\subseteq \I(\cF)$ such that $|\cB|<\dnumber$ and for each $A\in \I(\cF)$ there is $B\in \cB$ with $A\subset^* B$.
Let $\cD\subseteq\omega^\omega$ be cofinal in $\omega^\omega$ and $|\cD|=\dnumber$.

For each $g\in \cD$ we put $A^g=\{(n,g(n)): n\in A\}$.
Since $|(A^g)_{(n)}|\leq f(n)$ for every $n\in\omega$,  $A^g\in\I(\cF)$.
Hence there is $B^g\in\cB$ with $A^g\subseteq B^g$.
Let $n_g\in\omega$ be such that $(B^g)_{(n)}$ is finite for all $n\geq n_g$.
We define $h_g: A\to \omega$
by
$h_g(n)=\max((B^g)_{(n)})$ for $n\geq n_g$ and $h_g(n)=0$ otherwise.

Let $\cH=\{h_g:g\in\cD\}$.
Since $\cH$ is cofinal in $\cD\restriction A=\{g\restriction A:g\in\cD\}$ and $\cD\restriction A$ is cofinal in $\omega^A$,
$\cH$ is cofinal in $\omega^A$.
Thus, $|\cH|\geq \dnumber$.
On the other hand, $|\cH|\leq |\cB|<\dnumber$, a contradiction.
\end{proof}

It is known (see e.g.~\cite[Proposition~6.13]{MR2777744}) that
$\cofs(\I)\geq \cov(\cM)$
for analytic ideals $\I$ that are not countably generated.
Since $\cov(\cM)\leq \dnumber$, Proposition~\ref{prop:cofs-geq-d} gives a better bound for $\cofs(\I(\cF))$ in the case of analytic ideals of the form $\I(\cF)$.

It is known that 
$\cofs(\ED) = \continuum $ (see \cite[Theorem~1.6.4]{alcantara-phd-thesis}). Below we show that the same holds for ideals ``generated'' by one essentially nonzero function.

Recall that $\cF_f=\{k\cdot f:k\in\omega\}$ for  $f\in \omega^\omega$.

\begin{theorem}
\label{thm:cof-equals-c-for-one-function-generated}
$\cofs(\I(\cF_f))=\continuum$
for each $f\in\omega^\omega$
such that $f\neq^*0$.
\end{theorem}

\begin{proof}
Suppose, to the contrary, that $\cofs(\I(\cF_f))<\continuum$.
Then there is $\cB\subseteq\I(\cF_f)$ such that $|\cB|<\continuum$ and for every $A\in\I(\cF_f)$ there is $B\in\cB$ with $A\subseteq B$.

We define
$I_{n,0}= \{k\in\omega: 0\leq k < f(n)\}$
and
$I_{n,i}= \{k\in\omega: f(n)\cdot i\leq k < f(n)\cdot(i+1)\}$
for every $n,i\in\omega$, $i\geq 1$.

Let $\cA=\{A_\alpha:\alpha<\continuum\}$ be an almost disjoint family on $\omega$ (i.e.~$A_\alpha\cap A_\beta$ is finite for $\alpha\neq \beta$ and $A_\alpha\in[\omega]^\omega$).

By $\phi_\alpha:\omega\to A_\alpha$ we denote the increasing enumeration of the set $A_\alpha$.

For every $\alpha<\continuum$, we define
$$C_\alpha = \bigcup_{n\in \omega}\{n\}\times I_{n,\phi_\alpha(n)}$$
and note that $C_\alpha\in\I(\cF_f)$ for every $\alpha<\continuum$.

Now for every $B\in \cB$ let $C_B=\{\alpha<\continuum: C_\alpha\subseteq B\}$.
Since $C_\alpha\in\I(\cF_f)$, $\bigcup_{B\in\cB}C_B=\continuum$.

If we show that $C_B$ is finite for every $B\in\cB$ then $|\bigcup_{B\in\cB}C_B|\leq \omega\cdot |\cB|<\continuum$ and we obtain a contradiction that finishes the proof.

Let $B\in\cB$ and $k\in\omega$ be such that $|B_{(n)}|\leq kf(n)$ for all but finitely many $n$.
Then $|C_B|\leq k$.
Indeed, if $|C_B|>k$ then there is $D\in[\continuum]^{k+1}$ such that $C_\alpha\subseteq B$ for every $\alpha\in D$.
Since $\{A_\alpha:\alpha\in D\}$ is almost disjoint, the elements $\phi_{\alpha}(n)$, where $\alpha\in D$, are pairwise distinct for all but finitely many $n$.
Then the sets $(C_\alpha)_n$ where $\alpha\in D$ are pairwise disjoint for all but finitely many $n$.
Hence
$|B_{(n)}|\geq |D|\cdot |(C_\alpha)_{(n)}| = (k+1)\cdot f(n)$
for all but finitely many $n$.

Since $f(n)\neq 0$ for infinitely many $n$,
there are infinitely many $n$ with $|B_{(n)}|>kf(n)$, a contradiction.
\end{proof}

\begin{question}
\label{q:cof-for-ideals-generated-by-countably many functions}
Does $\cofs(\I(\cF))=\continuum$ for each countable family $\cF$ such that $\I(\cF)\neq\fin\otimes\{\emptyset\}$?
\end{question}


\subsection{Covering}

It is known that 
$\covs(\fin\otimes\fin)=\bnumber$ 
and
$\covs(\ED) = \non(\cM)$ 
(see \cite[Theorems~1.6.19 and 1.6.4]{alcantara-phd-thesis}).
Below we observe that $\bnumber$ is a lower bound 
and 
$\non(\cM)$ is an upper bound 
of the cardinal $\covs$ for every tall ideal of the form $\I(\cF)$.
In the proofs we will use the following lemma.
 
\begin{lemma}[{\cite[Proposition~3.1]{MR2319159}}]
\label{lem:basic-properties-of-covs}
Let $\I,\J$ be tall ideals on $\omega$.
\begin{enumerate}
\item
If $\I\leq_K\J$ then $\covs(\I)\geq \covs(\J)$.\label{lem:basic-properties-of-covs-Katetov-order}
\item
If $\I\subseteq \J$ then $\covs(\I)\geq \covs(\J)$.\label{lem:basic-properties-of-covs-subideal}
\item
If $X\notin\I$ then $\covs(\I) \geq  \covs(\I\restriction X)$.\label{lem:basic-properties-of-covs-restriction}
\end{enumerate}
\end{lemma}

\begin{proposition}
\label{prop:cov-leq-geq}
$\bnumber\leq \covs(\I(\cF))\leq \non(\cM)$ for each tall ideal $\I(\cF)$.
\end{proposition}

\begin{proof}
Using the inclusion $\I(\cF)\subseteq\fin\otimes\fin$
and  Proposition~\ref{lem:basic-properties-of-covs}(\ref{lem:basic-properties-of-covs-subideal}) we obtain  $\covs(\I(\cF))\geq \covs(\fin\otimes\fin)=\bnumber$.

Since $\I(\cF)$ is tall, there is $B\in[\omega]^\omega$ and $f\in \cF$ such that $f(n)\neq 0$ for every $n\in B$ (see Proposition~\ref{prop:IF-is-dense-characterization}).
Then $X=B\times \omega\notin\I(\cF)$
and
by Proposition~\ref{prop:characterization-of-known-examples-IF}(\ref{prop:characterization-of-known-examples-IF-ED-subset})
we have $\ED\restriction X\subseteq \I(\cF)\restriction X$.

Since $\ED\restriction X$ and $\ED$ are isomorphic, $\covs(\ED\restriction X) = \covs(\ED)=\non(\cM)$.

By Lemma~\ref{lem:basic-properties-of-covs}(\ref{lem:basic-properties-of-covs-subideal}) we have
$\non(\cM)=\covs(\ED\restriction X)\geq \covs(\I(\cF)\restriction X)$
and
by Lemma~\ref{lem:basic-properties-of-covs}(\ref{lem:basic-properties-of-covs-restriction}) we have 
$\covs(\I(\cF)\restriction X)\geq \covs(\I(\cF))$.
Thus,
$\covs(\I(\cF))\leq \non(\cM)$.
\end{proof}

It is easy to see that
$\covs(\I)\geq \pnumber$
for every tall ideal $\I$, where $\pnumber$ is the pseudo-intersection number. 
Since $\pnumber\leq \bnumber$, 
Proposition~\ref{prop:cov-leq-geq} gives a better lower bound for $\covs(\I(\cF))$ in the case of tall ideals of the form $\I(\cF)$.

In the sequel we will use the following characterization of the cardinal $\non(\cM)$.

\begin{lemma}[{see e.g.~\cite[Lemma~2.4.8]{MR1350295}}]
\label{lem:characterization-of-nonM}
$$\non(\cM)  = \min\{|\cS|:\cS\subseteq \cC\land \forall f\in\omega^\omega \, \exists S\in\cS \, \exists^\infty n\in\omega \, (f(n)\in S(n))\},$$
where
$$\cC=\left\{S\in\left([\omega]^{<\omega}\right)^\omega:\sum_{n\in\omega}\frac{|S(n)|}{(n+1)^2}<\infty\right\}.$$
\end{lemma}

\begin{proposition}
\label{prop:cov-geq-nonM}
Let $\I(\cF)$ be a tall ideal.
If there is an increasing sequence $n_0<n_1<\dots$ such that  
$$\sum_{i\in\omega} \frac{f(i)}{(n_i+1)^2}<\infty$$ for every $f\in\cF$ then $\covs(\I(\cF))= \non(\cM)$.
\end{proposition}

\begin{proof}
($\leq$) It follows from Proposition~\ref{prop:cov-leq-geq}.

($\geq$)
Let $\cA\subseteq \I(\cF)$ be a family of cardinality less than $\non(\cM)$. For every $A\in\cA$ there is a function $f_A\in\cF$ and $n_A\in\omega$ such that $|A_{(n)}|\leq f_A(n)$ for every $n\geq n_A$. For every $A\in\cA$ we define the function $S_A:\omega\to[\omega]^{<\omega}$ by
$$S_A(n)=
\begin{cases}
A_{(i)} & \text{for }  n=n_i \land i\geq n_A,\\
\emptyset & \text{otherwise}.
\end{cases}$$
Note that 
$$\sum_{n\in\omega}\frac{|S(n)|}{(n+1)^2}
=
\sum_{i \geq n_A}\frac{|S(n_i)|}{(n_i+1)^2}
=
\sum_{i \geq n_A}\frac{|A_{(i)}|}{(n_i+1)^2}
\leq
\sum_{i \geq n_A}\frac{f_A(i)}{(n_i+1)^2}
<\infty$$
and $|\{S_A:A\in\cA\}|\leq |\cA|<\non(\cM)$.
By Lemma~\ref{lem:characterization-of-nonM}
there is $g\in\omega^\omega$ such that for every $A\in\cA$ and almost all $n\in\omega$ we have $g(n)\not\in S_A(n)$. 

Let $h:\omega\to\omega$ be given by $h(k)=g(n_k)$ for every $k\in\omega$.

Then $B=(\omega\times\omega)\setminus h \in\cP(\omega\times\omega)\setminus \fin^*$
and for every $A\in\cA$ we have $A\subset^* B$.
\end{proof}

\begin{proposition}
\label{prop:cov-geq-nonM-for-less-than-b-family-F}
Let $\I(\cF)$ be a tall ideal.
If $|\cF|<\bnumber$, then $\covs(\I(\cF)) = \non(\cM)$.
\end{proposition}

\begin{proof}
It is not difficult to see that for every $f\in \cF$ there is an increasing sequence $\langle n^f_i \rangle$ such that 
$\sum_{i\in\omega} f(i) /(n^f_i+1)^2<\infty$.
Since $|\cF|<\bnumber$, there is an increasing sequence $\langle n_i\rangle$ such that $\langle  n^f_i\rangle \leq^* (n_i)$ for every $f\in \cF$.
Let $n^f\in\omega$ be such that $n^f_i\leq n_i$ for every $i\geq n^f$.
Then 
$$
\sum_{i\in\omega} \frac{f(i)}{(n_i+1)^2} 
\leq 
\sum_{i<n^f} \frac{f(i)}{(n_i+1)^2} 
+
\sum_{i\geq n^f} \frac{f(i)}{(n^f_i+1)^2}<\infty,$$
so
Proposition~\ref{prop:cov-geq-nonM} finishes the proof.
\end{proof}


\begin{acknowledgment}
The authors would like to thank the reviewers for their comments that helped to improve the manuscript.
\end{acknowledgment}


\bibliographystyle{amsplain}
\bibliography{paper}

\providecommand{\bysame}{\leavevmode\hbox to3em{\hrulefill}\thinspace}
\providecommand{\MR}{\relax\ifhmode\unskip\space\fi MR }
\providecommand{\MRhref}[2]{%
  \href{http://www.ams.org/mathscinet-getitem?mr=#1}{#2}
}
\providecommand{\href}[2]{#2}
\begin{thebibliography}{10}

\bibitem{MR3545234}
A.~Avil\'{e}s, V.~Kadets, A.~P\'{e}rez, and S.~Solecki, \emph{Baire theorem for
  ideals of sets}, J. Math. Anal. Appl. \textbf{445} (2017), no.~2, 1221--1231.
  \MR{3545234}

\bibitem{MR3712964}
M.~Balcerzak, M.~Pop\l{}awski and A.~Wachowicz, \emph{The {B}aire category of
  ideal convergent subseries and rearrangements}, Topology Appl. \textbf{231}
  (2017), 219--230. \MR{3712964}

\bibitem{MR1350295}
T.~Bartoszy\'{n}ski and H.~Judah, \emph{Set theory}, A K Peters, Ltd.,
  Wellesley, MA, 1995, On the structure of the real line. \MR{1350295}

\bibitem{MR3436368}
Piotr Borodulin-Nadzieja, Barnab\'{a}s Farkas, and Grzegorz Plebanek,
  \emph{Representations of ideals in {P}olish groups and in {B}anach spaces},
  J. Symb. Log. \textbf{80} (2015), no.~4, 1268--1289. \MR{3436368}

\bibitem{MR3247032}
J.~Brendle and D.~A. Mej\'{\i}a, \emph{Rothberger gaps in fragmented ideals},
  Fund. Math. \textbf{227} (2014), no.~1, 35--68. \MR{3247032}

\bibitem{MR1686797}
J.~Brendle and S.~Shelah, \emph{Ultrafilters on {$\omega$}---their ideals and
  their cardinal characteristics}, Trans. Amer. Math. Soc. \textbf{351} (1999),
  no.~7, 2643--2674. \MR{1686797}

\bibitem{MR2778559}
L.~Bukovsk\'{y}, \emph{The structure of the real line}, Instytut Matematyczny
  Polskiej Akademii Nauk. Monografie Matematyczne (New Series) [Mathematics
  Institute of the Polish Academy of Sciences. Mathematical Monographs (New
  Series)], vol.~71, Birkh\"auser/Springer Basel AG, Basel, 2011. \MR{2778559}

\bibitem{MR817590}
Jean Calbrix, \emph{Classes de {B}aire et espaces d'applications continues}, C.
  R. Acad. Sci. Paris S\'{e}r. I Math. \textbf{301} (1985), no.~16, 759--762.
  \MR{817590}

\bibitem{MR3490916}
David Chodounsk\'{y}, Osvaldo Guzm\'{a}n~Gonz\'{a}lez, and Michael
  Hru\v{s}\'{a}k, \emph{Mathias-{P}rikry and {L}aver type forcing; summable
  ideals, coideals, and +-selective filters}, Arch. Math. Logic \textbf{55}
  (2016), no.~3-4, 493--504. \MR{3490916}

\bibitem{MR3436375}
David Chodounsk\'{y}, Du\v{s}an Repov\v{s}, and Lyubomyr Zdomskyy,
  \emph{Mathias forcing and combinatorial covering properties of filters}, J.
  Symb. Log. \textbf{80} (2015), no.~4, 1398--1410. \MR{3436375}

\bibitem{MR2434680}
Pratulananda Das, Pavel Kostyrko, W\l{}adys\l{}aw Wilczy\'{n}ski, and Prasanta
  Malik, \emph{{$I$} and {$I^*$}-convergence of double sequences}, Math.
  Slovaca \textbf{58} (2008), no.~5, 605--620. \MR{2434680}

\bibitem{MR2520152}
G.~Debs and J.~Saint~Raymond, \emph{Filter descriptive classes of {B}orel
  functions}, Fund. Math. \textbf{204} (2009), no.~3, 189--213. \MR{2520152}

\bibitem{MR1711328}
I.~Farah, \emph{Analytic quotients: theory of liftings for quotients over
  analytic ideals on the integers}, Mem. Amer. Math. Soc. \textbf{148} (2000),
  no.~702, xvi+177. \MR{1711328}

\bibitem{MR2905404}
Rafa\l{} Filip\'{o}w, Nikodem Mro\.{z}ek, Ireneusz Rec\l{}aw, and Piotr Szuca,
  \emph{Ideal version of {R}amsey's theorem}, Czechoslovak Math. J.
  \textbf{61(136)} (2011), no.~2, 289--308. \MR{2905404}

\bibitem{MR4036731}
Jan Greb\'{\i}k and Michael Hru\v{s}\'{a}k, \emph{No minimal tall {B}orel ideal
  in the {K}at\v{e}tov order}, Fund. Math. \textbf{248} (2020), no.~2,
  135--145. \MR{4036731}

\bibitem{MR3513296}
Osvaldo Guzm\'{a}n-Gonz\'{a}lez and David Meza-Alc\'{a}ntara, \emph{Some
  structural aspects of the {K}at\v{e}tov order on {B}orel ideals}, Order
  \textbf{33} (2016), no.~2, 189--194. \MR{3513296}

\bibitem{MR2319159}
F.~Hern\'{a}ndez-Hern\'{a}ndez and M.~Hru\v{s}\'{a}k, \emph{Cardinal invariants
  of analytic {$P$}-ideals}, Canad. J. Math. \textbf{59} (2007), no.~3,
  575--595. \MR{2319159}

\bibitem{MR2777744}
M.~Hru\v{s}\'{a}k, \emph{Combinatorics of filters and ideals}, Set theory and
  its applications, Contemp. Math., vol. 533, Amer. Math. Soc., Providence, RI,
  2011, pp.~29--69. \MR{2777744}

\bibitem{MR3692233}
M.~Hru\v{s}\'{a}k, D.~Meza-Alc\'{a}ntara, E.~Th\"{u}mmel, and C.~Uzc\'{a}tegui,
  \emph{Ramsey type properties of ideals}, Ann. Pure Appl. Logic \textbf{168}
  (2017), no.~11, 2022--2049. \MR{3692233}

\bibitem{MR3696069}
Michael Hru\v{s}\'{a}k, \emph{Kat\v{e}tov order on {B}orel ideals}, Arch. Math.
  Logic \textbf{56} (2017), no.~7-8, 831--847. \MR{3696069}

\bibitem{MR3019575}
Michael Hru\v{s}\'{a}k and David Meza-Alc\'{a}ntara, \emph{Kat\v{e}tov order,
  {F}ubini property and {H}ausdorff ultrafilters}, Rend. Istit. Mat. Univ.
  Trieste \textbf{44} (2012), 503--511. \MR{3019575}

\bibitem{MR1844385}
Pavel Kostyrko, Tibor \v{S}al\'{a}t, and W\l{}adys\l{}aw Wilczy\'{n}ski,
  \emph{{$\mathcal{I}$}-convergence}, Real Anal. Exchange \textbf{26}
  (2000/01), no.~2, 669--685. \MR{1844385}

\bibitem{MR3351990}
A.~Kwela, \emph{A note on a new ideal}, J. Math. Anal. Appl. \textbf{430}
  (2015), no.~2, 932--949. \MR{3351990}

\bibitem{MR3784399}
\bysame, \emph{Ideal weak {QN}-spaces}, Topology Appl. \textbf{240} (2018),
  98--115. \MR{3784399}

\bibitem{MR2491780}
Mikl\'{o}s Laczkovich and Ireneusz Rec\l{}aw, \emph{Ideal limits of sequences
  of continuous functions}, Fund. Math. \textbf{203} (2009), no.~1, 39--46.
  \MR{2491780}

\bibitem{MR1708151}
Alain Louveau and Boban Velickovic, \emph{Analytic ideals and cofinal types},
  Ann. Pure Appl. Logic \textbf{99} (1999), no.~1-3, 171--195. \MR{1708151}

\bibitem{MR3016411}
Martin Ma\v{c}aj and Martin Sleziak,
  \emph{{$\mathcal{I}^{\mathcal{K}}$}-convergence}, Real Anal. Exchange
  \textbf{36} (2010/11), no.~1, 177--193. \MR{3016411}

\bibitem{alcantara-phd-thesis}
D.~Meza-Alc\'{a}ntara, \emph{Ideals and filters on countable set}, Ph.D.
  thesis, Universidad Nacional Aut\'{o}noma de M\'{e}xico, 2009.

\bibitem{MR3555332}
Hiroaki Minami and Hiroshi Sakai, \emph{Kat\v{e}tov and {K}at\v{e}tov-{B}lass
  orders on {$F_\sigma$}-ideals}, Arch. Math. Logic \textbf{55} (2016),
  no.~7-8, 883--898. \MR{3555332}

\bibitem{MR2002719}
Ferenc M\'{o}ricz, \emph{Statistical convergence of multiple sequences}, Arch.
  Math. (Basel) \textbf{81} (2003), no.~1, 82--89. \MR{2002719}

\bibitem{MR3550610}
Nikodem Mro\.{z}ek, \emph{Some applications of the {K}at\v{e}tov order on
  {B}orel ideals}, Bull. Pol. Acad. Sci. Math. \textbf{64} (2016), no.~1,
  21--28. \MR{3550610}

\bibitem{MR3453774}
Tomasz Natkaniec and Piotr Szuca, \emph{On the ideal convergence of sequences
  of quasi-continuous functions}, Fund. Math. \textbf{232} (2016), no.~3,
  269--280. \MR{3453774}

\bibitem{MR1511092}
Alfred Pringsheim, \emph{Zur {T}heorie der zweifach unendlichen
  {Z}ahlenfolgen}, Math. Ann. \textbf{53} (1900), no.~3, 289--321. \MR{1511092}

\bibitem{MR3778962}
Hiroshi Sakai, \emph{On {K}at\v{e}tov and {K}at\v{e}tov-{B}lass orders on
  analytic {P}-ideals and {B}orel ideals}, Arch. Math. Logic \textbf{57}
  (2018), no.~3-4, 317--327. \MR{3778962}

\bibitem{MR540490}
S.~Shelah and M.~E. Rudin, \emph{Unordered types of ultrafilters}, Topology
  Proc. \textbf{3} (1978), no.~1, 199--204 (1979). \MR{540490}

\bibitem{MR1758325}
S.~Solecki, \emph{Filters and sequences}, Fund. Math. \textbf{163} (2000),
  no.~3, 215--228. \MR{1758325}

\bibitem{MR579439}
M.~Talagrand, \emph{Compacts de fonctions mesurables et filtres non
  mesurables}, Studia Math. \textbf{67} (1980), no.~1, 13--43. \MR{579439}

\bibitem{MR1321463}
Fons van Engelen, \emph{On {B}orel ideals}, Ann. Pure Appl. Logic \textbf{70}
  (1994), no.~2, 177--203. \MR{1321463}

\bibitem{MR1004056}
S.~Zafrany, \emph{Borel ideals vs.\ {B}orel sets of countable relations and
  trees}, Ann. Pure Appl. Logic \textbf{43} (1989), no.~2, 161--195.
  \MR{1004056}

\end{thebibliography}

\end{document}